\documentclass[12pt]{article}
\usepackage[utf8]{inputenc}
\usepackage[a4paper, total={6in, 8in}, left=1.134in]{geometry}
\usepackage{amsmath}
\usepackage{amssymb}
\usepackage{amsthm}
\usepackage{hyperref}
\usepackage{tikz-cd}
\hypersetup{colorlinks=true,linkcolor=black,citecolor=black}

\newtheorem{theorem}{Theorem}[section]
\newtheorem{lemma}[theorem]{Lemma}
\newtheorem{proposition}[theorem]{Proposition}
\newtheorem{corollary}[theorem]{Corollary}
\newtheorem{definition}[theorem]{Definition}
\newtheorem{remark}[theorem]{Remark}

\numberwithin{equation}{section}

\newcommand{\R}{\mathbb{R}}

\newcommand{\Q}{\mathbb{Q}}

\renewcommand{\P}{\mathbb{P}}
\newcommand{\CP}{\mathbb{CP}}
\newcommand{\OO}{\mathcal{O}}

\renewcommand{\d}{\partial}
\newcommand{\dbar}{\overline{\partial}}

\newcommand{\e}{\varepsilon}

\DeclareMathOperator{\Ric}{Ric}

\DeclareMathOperator{\Kah}{Kah}

\DeclareMathOperator{\PSH}{PSH}

\newcommand{\oo}[1]{\overline{#1}}

\newcommand{\ip}[1]{\left\langle#1\right\rangle}
\newcommand{\abs}[1]{\left|#1\right|}

\title{Fano Fibrations and Twisted K\"ahler-Einstein Metrics I}
\author{Alexander Bednarek\\The University of Sydney\\\texttt{alexander.bednarek@sydney.edu.au}}
\date{\today}

\begin{document}

\maketitle

\begin{abstract}
\noindent This is the first of two papers studying both the geometric structure of Fano fibrations and the application to K\"ahler-Ricci flows developing a singularity in finite time. Given a Fano fibration which is generated by Kawamata's theorem from a compact K\"ahler manifold $X$ endowed with an ample, rational line bundle $L$ and non-nef canonical line bundle $K_X$, we construct a $(1,1)$-form on the regular part of the base analytic variety which is related to the Weil-Petersson metric. It is also proven that the singular K\"ahler metric constructed by Zhang, Zhang, on the base analytic variety satisfies a twisted K\"ahler-Einstein equation involving this $(1,1)$-form and, for a submersion, that the Chern classes of $X$ and the base manifold decompose in terms of this $(1,1)$-form.
\end{abstract}

\section{Introduction}

\noindent Fano fibrations are deeply connected with finite time singularities of the K\"ahler-Ricci flow. When the K\"ahler-Ricci flow is run on a compact K\"ahler manifold, whether or not it develops a singularity in a finite time is a condition solely dependent on the first Chern class. If a finite time singularity is encountered and the K\"ahler manifold is equipped with rational, ample, line bundle, or equivalently a rational, initial metric then the manifold naturally admits the structure of a Fano fibration over an analytic variety in projective space.\\
The Analytic Minimal Model Program proposed by Song and Tian, \cite{BEG13,JST23,ST07,ST12}, puts this geometric picture into an larger scheme to classify compact K\"ahler manifolds using the K\"ahler-Ricci flow. One behaviour that the Analytic Minimal Model Program predicts is that the K\"ahler-Ricci flow collapses the Fano fibres to points and the manifold itself collapses onto the base space of the fibration \cite{JST23}. This paper studies the geometry of Fano fibrations in an attempt to understand the right model for this collapsing behaviour of K\"ahler-Ricci flow. In particular, we are interested in forms related to the deformation theory of the complex structures of the Fano fibres, and the existence of canonical singular K\"ahler metrics on the base space of the fibration. In the second paper, \cite{B25}, of this two part series, we study the applications of some of the results of this paper to the collapsing behaviour predicted for the K\"ahler-Ricci flow by the Analytic Minimal Model Program.\\

\noindent We will assume $(X,\omega_0)$ is a compact K\"ahler manifold where $\omega_0$ is a K\"ahler metric satisfying $[\omega_0] = 2\pi c_1(L)$ for some rational line bundle $L$ and assume that $K_X$ is not nef. Then Kawamata's basepoint free theorem generates a Fano fibration, i.e., a proper, holomorphic map with connected fibres
\begin{equation*}
    f: X \to B \subseteq \CP^N
\end{equation*}
where $B$ is an irreducible, normal, analytic variety. Let us denote the singular part of the variety together with the critical points of the map $f$ by $S'$, and the smooth Fano fibres by $X_b := f^{-1}(b)$ for all $b \in B \backslash S'$. Also set $\lambda$ to be the positive constant such that $2 \pi c_1(X_b) = \lambda [\omega_{0,b}]$ where we denote $\omega_{0,b} = \omega_0|_{X_b}$. We also assume that the fibres are of intermediate dimension, i.e. $0 < \dim B < \dim X$.\\

\noindent The fibration naturally gives us a family of diffeomorphic K\"ahler manifolds with varying complex structures. The Kodaira-Spencer map $\rho_b: T^{1,0}_b B \to H^{1}(X_b,T^{1,0}X_b)$ encapsulates this deformation of complex structure for a fixed $b \in B \backslash S'$, \cite{H05,K86,V02,T89}, and is used to define the Weil-Petersson metric. In particular, the Weil-Petersson metric is defined as the $L^2$-inner product of harmonic forms representing the image of the Kodaira-Spencer map, \cite{FS90, T87}.\\ 
Tian's result, \cite[Thm. 2]{T87}, for the Weil-Petersson metric in the deformation theory for Calabi-Yau manifolds states that the Weil-Petersson metric can be locally expressed as
\begin{equation*}
    \omega_{WP} = -i\d\dbar\log\left(\int_{X_b} \Omega_{\Psi_b}\right)
\end{equation*}
using a local family of volume forms $\Omega_{\Psi_b}$ on $X_b$ with $\Ric \Omega_{\Psi_b} = 0$. This motivates the following definition.

\begin{theorem}
There exists a $(1,1)$-form $\omega_{WP,\lambda}$ on $B \backslash S'$ defined locally in the open set $U \subseteq B \backslash S'$ by
\begin{equation*}
    \omega_{WP,\lambda}|_{U} = -i\d\dbar \log \left(\int_{X_b} \Omega_{\Psi_b}\right),
\end{equation*}
where $\Omega_{\Psi_b}$ is a local family of volume forms on $X_b$ over $U$ satisfying $\Ric \Omega_{\Psi_b} = \lambda \omega_{0,b}$.
\end{theorem}

\noindent It is known due to the work of Song and Tian, \cite{ST07, ST12,T19}, that given a Calabi-Yau fibration, there exists a singular K\"ahler metric $\omega_B$ on the base analytic variety which satisfies the twisted K\"ahler-Einstein equation
\begin{equation*}
    \Ric \omega_B = -\omega_B + \omega_{WP}
\end{equation*}
on $B\backslash S'$. Under the assumption of the abundance conjecture, this geometric structure always appears in the case that the K\"ahler-Ricci flow develops a singularity after infinite time. Tosatti, Weinkove, Yang, \cite{TWY18,T19}, prove that the K\"ahler-Ricci flow converges in the $C^0_{loc}(X\backslash S)$-topology to the pullback of this twisted K\"ahler-Einstein metric $\omega_B$. Moreover, it was recently proven by Zhang, Zhang, that one can use the argument of Song and Tian in \cite{ST12} to construct a singular K\"ahler metric on the base analytic variety of the Fano fibration, \cite{ZZ20}. Furthermore, the continuity method also converges in the $C^0_{loc}(X\backslash S)$-topology to the pullback of this singular K\"ahler-metric, \cite{ZZ20}. The convergence in the setting of the K\"ahler-Ricci flow will be discussed in the second paper, \cite{B25}. Below, we prove that the metric of Zhang and Zhang satisfies a twisted K\"ahler-Einstein equation involving $\omega_{WP,\lambda}$.

\begin{theorem}
For the Fano fibration $f:X \to B$, the base space metric $\omega_B$ of Zhang, Zhang, \cite{ZZ20}, satisfies the following twisted K\"ahler-Einstein equation
\begin{equation*}
    \Ric \omega_{B} = -\omega_B -\lambda \eta +\omega_{WP,\lambda}
\end{equation*}
on $B \backslash S'$, where $\eta$ is a multiple of the Fubini-Study metric. Furthermore, there exists a metric on the base space, $\omega_{B}'$, satisfying the following twisted K\"ahler-Einstein equation
\begin{equation*}
    \Ric \omega_B' = -\omega_B' + \omega_{WP,\lambda}
\end{equation*}
on $B \backslash S'$. The K\"ahler metrics $\omega_B$ and $\omega_B'$ are both smooth on the $B\backslash S'$ and continuous on $B$.
\end{theorem}

\noindent An immediate consequence of these twisted K\"ahler-Einstein equations is the following decomposition of the Chern classes in terms of the $(1,1)$-form $\omega_{WP,\lambda}$. We also present an algebro-geometric proof.

\begin{corollary}
If $B$ is a smooth K\"ahler manifold and $f: X \to B$ is a submersion, then we have
\begin{equation*}
    -2\pi c_1(B) + [\omega_{WP,\lambda}] = (\lambda + 1)[\eta]
\end{equation*}
and
\begin{equation*}
    2 \pi c_1(X) = \lambda [\omega_0] + 2 \pi f^* c_1(B) - f^*[\omega_{WP,\lambda}].
\end{equation*}
\end{corollary}

\noindent We also comment that all of these results and constructions can and are adapted to the case where we assume the Fano fibration admits a smoothly varying family of K\"ahler-Einstein metrics. This will be of interest in discussing the convergence of the K\"ahler-Ricci flow to the twisted K\"ahler-Einstein metric in the second paper \cite{B25}.\\

\noindent \textbf{Acknowledgements.} \textit{The author would like to thank his supervisors Zhou Zhang and Haotian Wu, for the many insightful discussions and their support. His gratitude also goes to Tiernan Cartwright, Xilun Li, Minghao Miao and Xiaohua Zhu, for the discussions concerning some details in the arguments and their interest. Last but not least, the author would like to thank Wangjian Jian, Zhenlei Zhang and Gang Tian, for extending an invitation to visit BICMR and for many discussions, during which a large proportion of this paper was completed. The research visit was supported by the Chinese Academy of Sciences and the Beijing International Center for Mathematical Research.}

\section{The Semi-Prescribed Ricci Curvature Form}

Let $(X, \omega_0)$ be a compact K\"ahler manifold of dimension $\dim X = n$. We consider the following normalization of the K\"ahler-Ricci flow
\begin{equation}\label{KRF}\tag{KRF}
    \frac{\d \omega(t)}{\d t} = -\Ric (\omega(t)) - \omega(t), \quad \omega(0) = \omega_0,  
\end{equation}
for which the maximal existence time is given by the value
\begin{equation*}
    T = \sup\{t\in [0,\infty) : e^{-t}[\omega_0] - (1-e^{-t}) 2\pi c_1(X) \in \Kah(X)\},
\end{equation*}
by a well-known result of Tian and Zhang, \cite{TZ06}.\\
Let us assume $T < \infty$, which is equivalent to the case that $K_X$ is not nef, and we shall also assume that the initial metric is rational, i.e., $[\omega_0] \in H^{1,1}(X,\Q)$ and there exists a holomorphic $\Q$-line bundle $L$ such that $2\pi c_1(L) = [\omega_0]$. Kawamata's rationality and basepoint-free theorems imply
\begin{equation*}
    D := e^{-T} L + (1-e^{-T}) K_X
\end{equation*}
is a holomorphic $\Q$-line bundle which is semi-ample, i.e., for some sufficiently large integer $k \geq 1$, the line bundle $D$ induces a proper, surjective, holomorphic map with connected fibres, \cite{L04},
\begin{equation*}
    f: X \to B \subseteq \P H^0(X,kD)
\end{equation*}
where $kD = f^* \OO_B(1)$ and the image $B$ is an irreducible, normal, analytic variety. We shall fix $k$ in the construction of the $(1,1)$-form $\omega_{WP,\lambda}$. Let $\eta := \frac{1}{k} \omega_{FS}|_B$ be the restriction of the Fubini Study metric $\omega_{FS}$ to $B$, and note that $[f^*\eta] = [\omega(T)]$, where we abuse notation to mean $[\omega(T)] = e^{-T} [\omega_0] - (1-e^{-T}) 2\pi c_1(X)$ and so the cohomology class of the K\"ahler-Ricci flow (\ref{KRF}) evolves according to 
\begin{equation*}
    [\omega(t)] = \frac{e^{-t}-e^{-T}}{1-e^{-T}} [\omega_0] + \frac{1-e^{-t}}{1-e^{-T}} [f^*\eta].
\end{equation*}
Denote the Iitaka dimension by $m = \dim B$ and let $S' \subset B$ be the set containing the singular part of $B$ together with the critical points of $f$, and set $S = f^{-1}(S')$. As mentioned, we shall assume that $0 < m < n$.\\
\noindent We will denote the fibres of the map $f$ by $X_b := f^{-1}(b)$ for $b \in B \backslash S'$, and observe that $X_b$ is a smooth K\"ahler manifold equipped with K\"ahler metric $\omega_{0,b} := \omega_0|_{X_b}$. Moreover, $X_b$ is Fano as the restriction of the cohomology condition $[f^*\eta] = [\omega(T)]$ to each fibre $X_b$ implies
\begin{equation*}
    \lambda \omega_{0,b} \in 2\pi c_1(X_b)
\end{equation*}
for all $b \in B\backslash S'$, where $\lambda = \frac{e^{-T}}{1-e^{-T}} > 0$.\\
Following Zhang and Zhang, \cite{ZZ20}, (cf.\cite{Y78}),  one can construct a unique, $d$-closed, real, $(1,1)$-form $\omega_{SPR} = \omega_0 + i\d\dbar \rho_{SPR}$ on $X\backslash S$ for some $\rho_{SPR} \in C^\infty(X\backslash S, \R)$, which is a K\"ahler metric along each fibre $X_b$, denoted $\omega_{SPR,b}:= \omega_{SPR}|_{X_b}$, satisfying
\begin{equation*}
    \Ric(\omega_{SPR,b}) = \lambda\omega_{0,b}
\end{equation*}
for all $b \in B\backslash S'$. Let us refer to $\omega_{SPR}$ as the semi-prescribed Ricci curvature form since the restriction of $\omega_{SPR}$ to any fibre $X_b$ has Ricci curvature prescribed by $\lambda\omega_{0,b}$. This construction is analogous to that of the semi-Ricci flat form when $T = \infty$, $K_X$ is semi-ample and $X$ admits a Calabi-Yau fibration, \cite{ST07}.

\section{The $(1,1)$-Form $\omega_{WP,\lambda}$}

As $kD = f^* \OO_B(1)$, it follows that for every $b \in B\backslash S'$ we have
\begin{equation*}
    \OO_{X_b} = f^*\OO_B(1)|_{X_b} = kD|_{X_b} = kD_b,
\end{equation*}
as holomorphic line bundles, where for notational convenience we denote $D_b = D|_{X_b}$. Due to the rationality of $D$, there exists positive integers, $p,q,r$, such that
\begin{equation*}
    pD = qL + rK_X
\end{equation*}
and let us fix $k$ to be the smallest multiple of $p$ that generates the map $f$. Note $k' = k(1-e^{-T})$ is an integer. Therefore, after restricting to the fibre $X_b$, we have
\begin{equation*}
    D_b^k = L_b^\alpha \otimes K_{X_b}^\beta
\end{equation*}
where $\frac{\alpha}{\beta} = \frac{q}{r} = \lambda$, and we denote $L_b := L|_{X_b}$, and note $K_X|_{X_b} = K_{X_b}$ by the adjunction formula.\\
In the same way that we can define the relative canonical line bundle as $K_{X/B} = K_X - f^* K_B$, we define the associated relative $\Q$-line bundle, $D_{X/B}$, to $D$ by
\begin{equation}\label{associated_relative_bundle}
    D_{X/B} = \frac{1}{1-e^{-T}}D - f^* K_B = \lambda L +  K_{X/B}.
\end{equation}
We make the observation that
\begin{equation*}
k'D_{X/B}|_{X_b} = kD_b = \OO_{X_b},
\end{equation*}
so that $h^0(X_b, D_{X/B}^{k'}|_{X_b}) = 1$ is constant with respect to $b$, and we can apply \cite[Thm I.8.5.iii]{BHPV04} so that $f_* D_{X/B}^{k'}$ is a locally free, coherent sheaf of rank $1$, i.e., a holomorphic line bundle.\\
Let $\Psi$ be a local, non-vanishing, holomorphic section of $f_* D^{k'}_{X/B}$ on some open set $U' \subseteq B\backslash S'$. Then,
\begin{equation*}
    \Psi \in \Gamma(U', f_* D_{X/B}^{k'}) \iff \Psi \in \Gamma(f^{-1}(U'), D^{k'}_{X/B})
    \iff \Psi_b \in \Gamma(X_b, D_b^k)
\end{equation*}
is a family of global, non-vanishing, holomorphic sections of $D^k_b$ which vary holomorphically in $b \in U'$.\\
Let $(z_1,\dots,z_n)$ be local product coordinates in some suitable open set $U \subseteq X\backslash S$, where the first $n-m$ coordinates are in the fibre direction and let $s$ be a local, non-vanishing, holomorphic frame field for $L$ over $U$ and denote its restriction to each fibre, $X_b$, by $s_b := s|_{X_b}$. Therefore, we can write the family, $\Psi_b$, in local product coordinates as
\begin{equation*}
    \Psi_b = F(b,z) (s_b)^\alpha \otimes (dz^1 \wedge \dots \wedge dz^{n-m})^{\beta},
\end{equation*}
for some non-vanishing, holomorphic function $F(b,z)$, as $D^k_b = L^\alpha_b \otimes K_{X_b}^\beta$.\\
Next, we define a volume form $\Omega_{\Psi_b}$ on the fibre $X_b$ which is dependent on the choice of sections $\Psi_b$, however we need a Hermitian metric to contract the local section $s$. 
Recall that as $2\pi c_1(L) = [\omega_0]$, there exists a Hermitian metric $h_L$ such that $\Ric h_L = \omega_0$, which is unique up to a multiplicative constant which we fix arbitrarily, \cite[pg. 148]{GH94}. This constant only has the effect of scaling the entire family $\Omega_{\Psi_b}$ and leaves the value of $\Ric \Omega_{\Psi_b}$ unchanged as well as the definition of the $(1,1)$-form $\omega_{WP,\lambda}$. We take $h_{L_b} := h_L|_{X_b}$ and define $\Omega_{\Psi_b}$ in local coordinates to be
\begin{equation*}
    \Omega_{\Psi_b} = i^{n-m} |F(b,z)|^{\frac{2}{\beta}} |s_b|_{h_{L_b}}^{\frac{2\alpha}{\beta}} dz^1 \wedge \dots \wedge d\oo{z}^{n-m}.
\end{equation*}
This is well-defined and independent of local coordinates, in particular, it does not depend on the choice of local, non-vanishing, holomorphic section $s$. If $s'$ is another local, non-vanishing, holomorphic section of $L$ over $U$, then $s' = hs$ for some local, non-vanishing, holomorphic function $h$ and writing
\begin{equation*}
    \Psi_b = F'(b,z) (s'_b)^\alpha \otimes (dz^1 \wedge \dots \wedge dz^{n-m})^\beta,
\end{equation*}
we see $F(b,z) = h(b,z)^\alpha F'(b,z)$, which implies
\begin{equation*}
    i^{n-m}|F'(b,z)|^{\frac{2}{\beta}}|s'_b|_{h_{L_b}}^{\frac{2\alpha}{\beta}} dz^1 \wedge \dots \wedge d\oo{z}^{n-m} = i^{n-m}|F(b,z)|^{\frac{2}{\beta}}|s_b|_{h_{L_b}}^{\frac{2\alpha}{\beta}} dz^1 \wedge \dots \wedge d\oo{z}^{n-m}.
\end{equation*}
Alternatively, we can write $\Omega_{\Psi_b} = i^{n-m}((\Psi_b \wedge \oo{\Psi}_b) \llcorner h_{L^{\alpha}})^{1/\beta}$, where $h_{L^{\alpha}}$ is the induced Hermitian metric on $L^\alpha$, i.e., $h_{L^{\alpha}}(s^\alpha, \oo{s}^\alpha) = h_{L}(s,\oo{s})^\alpha$, and $\llcorner$ denotes contraction of the $L^\alpha \otimes \oo{L}^\alpha$-valued component of $\Psi_b \wedge \oo{\Psi}_b$ with the metric $h_{L}$. It follows from $\lambda = \frac{\alpha}{\beta}$ and the holomorphicity of $F(b,z)$ that
\begin{align*}
    \Ric \Omega_{\Psi_b} &= -i\d_{X_b}\dbar_{X_b}\log(|F(b,z)|^{\frac{2}{\beta}} |s_b|_{h_{L_b}}^{2\lambda})\\
    &= -\lambda i\d_{X_b}\dbar_{X_b} \log |s_b|_{h_{L_b}}^2 = \lambda \Ric h_{L}|_{X_b} = \lambda \omega_{0,b},
\end{align*}
from which we reconfirm $2 \pi c_1(X_b) = \lambda [\omega_{0,b}]$.\\
Finally, we define the $(1,1)$-form $\omega_{WP,\lambda}$, locally in $U'$, by
\begin{equation*}
    \omega_{WP,\lambda}|_{U'} = -i\d\dbar \log \int_{X_b} \Omega_{\Psi_b}.
\end{equation*}
To see the form $\omega_{WP,\lambda}$ can be globally defined, we can check it is independent of the choice of family of sections $\Psi_b$ and glues on overlapping open sets $U', V' \subseteq B \backslash S'$. Suppose $\Psi'_b \in H^0(X_b, D_b^k)$ is another family of global, non-vanishing, holomorphic sections over $X_b$ which varies holomorphically in $b \in V'$, for some open set $V' \subseteq B\backslash S'$ where $U' \cap V' \neq \varnothing$. Suppose $b \in U' \cap V'$, and write $\Psi_b'$ in local product coordinates as
\begin{equation*}
    \Psi'_b = F'(b,z) (s_b)^\alpha \otimes (dz^1 \wedge \dots \wedge dz^{n-m})^\beta,
\end{equation*}
so that the associated volume form $\Omega_{\Psi'_b}$ is given by
\begin{equation*}
    \Omega_{\Psi'_b} = i^{n-m} |F'(b,z)|^{\frac{2}{\beta}} |s_b|_{h_{L_b}}^{\frac{2\alpha}{\beta}} dz^1 \wedge \dots \wedge d\oo{z}^{n-m}.
\end{equation*}
As both $\Psi_b$ and $\Psi'_b$ are global, non-vanishing, holomorphic sections of $D_b$ over $X_b$, or equivalently, $\Psi$ and $\Psi'$ are local, non-vanishing, holomorphic sections of the line bundle $f_* D^{k'}_{X/B}$ over $U' \cap V'$, there exists a non-vanishing, holomorphic function $h$ over $U' \cap V'$ such that $\Psi' = h\Psi$. It follows that
\begin{equation*}
    F'(b,z) = h(b) F(b,z),
\end{equation*}
and
\begin{equation*}
    \Omega_{\Psi'_b} = |h|^{\frac{2}{\beta}} \Omega_{\Psi_b}.
\end{equation*}
Therefore, as $i\d\dbar \log |h|^{\frac{2}{\beta}} = 0$, we have
\begin{equation*}
    \omega_{WP,\lambda}|_{U' \cap V'} = -i\d\dbar\log \int_{X_b} \Omega_{\Psi'_b} = -i\d\dbar \log \int_{X_b} \Omega_{\Psi_b},
\end{equation*}
i.e., the local definitions of the $(1,1)$-form $\omega_{WP,\lambda}$ glue together to define a global form on $B\backslash S'$.\\
Furthermore, we can define a singular Hermitian metric $h_{WP,\lambda}$ on $f_* D_{X/B}^{k'}$ over $B \backslash S'$ by
\begin{equation*}
    \abs{\Psi}^2_{h_{WP,\lambda}} = \int_{X_b} \Omega_{\Psi_b},
\end{equation*}
so that the $(1,1)$-form $\omega_{WP,\lambda}$ is related to its curvature, $\Ric h_{WP,\lambda}$, by
\begin{equation*}
    \omega_{WP,\lambda} = \Ric h_{WP,\lambda}.
\end{equation*}

\noindent The notation $\omega_{WP,\lambda}$ is inspired by a theorem of Tian, \cite[Thm. 2]{T87}, on the Weil-Petersson metric on the universal polarized deformation space of a Calabi-Yau manifold. In the setting of infinite time singularities of the K\"ahler-Ricci flow with semi-ample $K_X$, we have exactly a map $f: X \to B$ which is a metrically polarized deformation of Calabi-Yau manifolds excluding singular fibres. A metrically polarized deformation of Calabi-Yau manifolds has a smoothly varying family of K\"ahler-metrics, i.e., $(\omega_{0,b})_{b \in B \backslash S'}$ associated to it, \cite[Defn. 3.2]{FS90}, and we denote the unique Ricci-flat metric in the class $[\omega_{0,b}]$ by $\omega_{SRF,b}$. This family of Ricci-flat metrics glues together smoothly to give the semi Ricci-flat form $\omega_{SRF}$ on $X \backslash S$ where $\omega_{SRF}|_{X_b} := \omega_{SRF,b}$, \cite{ST07}, \cite[Lem. 3.2]{ST12}.\\
We define the Weil-Petersson metric, $\omega_{WP}$, in the following way. Take the Kodaira-Spencer map
\begin{equation*}
    \rho_b: T^{1,0}_b B \to H^{1}(X_b,T^{1,0}X_b)
\end{equation*}
which measures the infinitesimal variation of complex structure of the fibers $X_b$ in the direction $V \in T^{1,0}_b B$. For example, if $b \mapsto \dim H^{1}(X_b,T^{1,0}X_b)$ is constant and $\rho_b = 0$ then $f:X \to B$ is locally biholomorphic to a product manifold, \cite[Thm. 4.6]{K86}. Then for any $V, W \in T^{1,0}_b B$ choose $\gamma_V$, $\gamma_W$, to be the unique $\Delta_{\dbar}$-harmonic representatives of $\rho_b(V), \rho_b(W)$, respectively. We define the Weil-Petersson metric to be
\begin{equation*}
    \omega_{WP}(V,\oo{W}) = \int_{X_b} \ip{\gamma_V, \oo{\gamma_W}}_{\omega_{SRF,b}} \omega_{SRF,b}^{n-m}. 
\end{equation*}
Tian proves that for a polarized deformation of a Calabi-Yau manifold,
\begin{equation*}
    \omega_{WP} = -i\d\dbar \log \int_{X_b} \Omega_{\Psi_b}
\end{equation*}
where the volume form $\Omega_{\Psi_b}$ is constructed with a similar method to this section using holomorphic $(n-m,0)$-forms $\Psi_b$ on $X_b$, \cite[Thm. 2]{T87}. In fact, the Weil-Petersson metric can be defined for any metrically polarized deformation of K\"ahler manifolds but, at the moment, an explicit connection between $\omega_{WP,\lambda}$ and $\omega_{WP}$ for our metrically polarized deformation of Fano manifolds is difficult to establish and may appear in future work.

\section{The Twisted K\"ahler-Einstein Metrics}

\noindent Again, by following Zhang and Zhang, \cite{ZZ20}, one can construct a nice metric on the base space. As $[f^*\eta] = [\omega(T)]$, we can let $\chi$ be the smooth $(1,1)$-form such that
\begin{equation}\label{choice_of_Omega}
    f^*\eta = e^{-T} \omega_0 + (1-e^{-T}) \chi
\end{equation}
and choose a smooth volume form $\Omega$ such that $\chi = -\Ric \Omega$. The volume form $\Omega$ is unique up to multiplicative constant which we fix by imposing
\begin{equation*}
    \int_X \Omega =  \int_X {n\choose m} \omega_0^{n-m} \wedge f^*\eta^m.
\end{equation*}
Recall the following definition \cite[Defn. 3.1]{ST12}.

\begin{definition}
The push-forward of a volume form $\Omega$ on $X$ by the map $f: X \to B$ is defined to be the $(m,m)$-current $f_* \Omega$ such that
\begin{equation*}
    \int_B \psi(f_* \Omega) = \int_X (f^*\psi) \Omega
\end{equation*}
for all continuous functions $\psi \in C^0(B)$.
\end{definition}
\noindent In particular, we have
\begin{equation*}
    \int_B \psi (f_*\Omega) = \int_X (f^*\psi) \Omega = \int_B \psi \int_{X_b}\Omega,
\end{equation*}
so that on $B\backslash S'$
\begin{equation*}
    f_*\Omega = \int_{X_b} \Omega
\end{equation*}
is a smooth volume form.\\
So let $G'$ be the smooth positive function on $B\backslash S'$ defined by
\begin{equation*}
    G' = \frac{f_* \Omega}{V\eta^m}
\end{equation*}
where $V = {n\choose m} \int_{X_b} \omega_{0,b}^{n-m}$ is a constant. This follows from the commutability of the push-forward and the exterior derivative, i.e.,
\begin{equation*}
    d(f_*\omega_{0}^{n-m}) = f_*(d \omega_{0}^{n-m}) = 0.
\end{equation*}

\noindent By \cite[Prop. 3.2, Thm. 3.2]{ST12}, there exists positive constants $\delta, \varepsilon > 0$ such that $0 < \delta \leq G' \in L^{1+\e}(B, \eta)$, and the complex Monge-Ampere equation
\begin{equation}\label{base_space_CMA}
    \left(\eta + i\d\dbar \rho_B\right)^m = G' e^{\rho_B} \eta^m
\end{equation}
has a unique solution $\rho_B \in \PSH(B,\eta) \cap C^0(B) \cap C^{\infty}(B\backslash S')$ such that
\begin{equation*}
    \omega_B := \eta + i\d\dbar \rho_B
\end{equation*}
is a smooth K\"ahler metric on $B \backslash S'$ which we shall call the base space metric.

\begin{remark}\label{base_space_CMA_remark}
The argument for existence and uniqueness is given by Song and Tian in \cite[Thm. 3.2]{ST12}. In summary, one uses a resolution of singularities for $B$, i.e. we choose a birational map $\pi: Y \to B$ where $Y$ is a smooth, analytic variety, and $E = \pi^{-1}(S')$ is a simple normal crossings divisor and lift the complex Monge-Ampere equation (\ref{base_space_CMA}) to $Y$. Next, one chooses a sequence of functions $G_j$ on $Y$ which converge to $\pi^*G'$ in $L^{1+\e}(Y,\pi^*\eta^m)$ and satisfy certain properties. This only relies on the fact that $0 < \delta \leq G' \in L^{1+\e}(B,\eta)$ for any K\"ahler metric $\eta$. The $G_j$ are used to define a family of CMA equations and after deriving $C^0$, $C^2$ and $C^k$-estimates for their solutions $\rho_j$ by standard arguments from Yau and elliptic PDE theory, independent of $j$, it is deduced the that solutions $\rho_j$ converge to a solution $\rho_Y \in \PSH(Y,\pi^*\eta) \cap C^0(Y) \cap C^{\infty}(Y\backslash E)$. Finally, $\rho_Y$ descends to $B$ to give us $\rho_B$.
\end{remark}

\noindent Now we define a real, smooth, positive function on $X\backslash S$ by
\begin{equation*}
    G = \frac{\Omega}{{n\choose m}\omega_{SPR}^{n-m} \wedge f^*\eta^m},
\end{equation*}
which is constant along fibres and descends to a smooth, positive function $G''$ on the regular part of the base space $B\backslash S'$.

\begin{lemma}\label{pullback_G'_is_G}
We have $G = f^*G'$ on $X\backslash S$.
\end{lemma}

\begin{proof}
The argument is similar to \cite[Lem. 3.3]{ST12}. In particular, for any function $\psi \in C^{0}(B)$,
\begin{align*}
    \int_{B\backslash S'} \psi G' \eta^m = \frac{1}{V}\int_{B\backslash S'} \psi (f_*\Omega) &= \frac{1}{V}\int_{B\backslash S'} \psi \left(\int_{X_b}\Omega\right)\\
    &= \frac{1}{V}\int_{X\backslash S} (f^*\psi) \Omega\\
    &= \frac{1}{V} \int_{X\backslash S} (f^*\psi) (f^*G'') {n\choose m} \omega_{SPR}^{n-m} \wedge f^*\eta^m\\
    &= \frac{1}{V} \int_{B\backslash S'} \psi G'' \left(\int_{X_b}\omega_{SPR,b}^{n-m} \right) \eta^m\\
    &= \int_{B\backslash S'} \psi G'' \eta^m.
\end{align*}
Thus $G' = G''$ on $B \backslash S'$ and the result follows.
\end{proof}

\noindent Let us prove the following very useful tool.

\begin{proposition}\label{ric_of_fibration_volume_form}
Let $\theta$ be any K\"ahler metric on $B\backslash S'$, then on $X \backslash S$ we have
\begin{equation}
    \Ric (\omega_{SPR}^{n-m} \wedge f^* \theta^m) = \lambda \omega_0  +f^* \Ric \theta - f^* \omega_{WP,\lambda}.
\end{equation}
\end{proposition}

\begin{proof}
Let $\Psi_b \in H^0(X_b,D_b^k)$ be a family of holomorphic sections which vary holomorphically in $b$ in some suitable open set $U' \subseteq B$, and are globally non-vanishing on their respective fibres. Then as the associated volume form $\Omega_{\Psi_b}$ satisfies $\Ric \Omega_{\Psi_b} = \lambda \omega_{0,b}$, as does $\omega_{SPR,b}^{n-m}$, so these volume forms differ by a multiplicative constant on each fiber. That is, 
\begin{equation*}
    \Omega_{\Psi_b} = f^*\mu \ \omega_{SPR,b}^{n-m},
\end{equation*}
where $\mu$ is a smooth function in $b \in U'$. Moreover, since $f^*\mu$ is constant along fibres, we have 
\begin{equation*}
    \mu= \frac{\int_{X_b} \Omega_{\Psi_b}}{\int_{X_b}\omega_{SPR,b}^{n-m}}
\end{equation*}
and furthermore,
\begin{equation*}
    \omega_{WP,\lambda}|_{U'} = -i\d\dbar \log \mu,
\end{equation*}
as $b \mapsto \int_{X_b} \omega_{SPR,b}^{n-m}$ is constant.\\ 
Now, note that
\begin{equation*}
    f^*\mu = \frac{\Omega_{\Psi_b}}{\omega_{SPR,b}^{n-m}} = \frac{\Omega_{\Psi_b} \wedge f^* \theta^m}{\omega_{SPR}^{n-m} \wedge f^* \theta^m}.
\end{equation*}
Additionally, using local product coordinates where the first $n-m$ coordinates represent the fibre direction again, the K\"ahler metric $\theta$ is given locally by
\begin{equation*}
    \theta = i\sum_{i,j=n-m+1}^n \theta_{i\oo{j}} dz^i \wedge d\oo{z}^j.
\end{equation*}
It follows that
\begin{equation*}
    \Omega_{\Psi_b} \wedge f^* \theta^m = i^n |F|^{\frac{2}{\beta}} |s|^{2\lambda}_{h_{L}} f^*\det(\theta_{i\oo{j}}) dz^1 \wedge \dots \wedge d\oo{z}^n
\end{equation*}
and so
\begin{align*}
    -i\d\dbar\log (\omega_{SPR}^{n-m} \wedge f^* \theta^m) &= -i\d\dbar \log \left(\frac{1}{f^*\mu} \Omega_{\Psi_b} \wedge f^* \theta^m\right)\\
    &= -i\d\dbar\log\left(\frac{|F|^{\frac{2}{\beta}} |s|^{2\lambda}_{h_{L}} f^*\det(\theta_{i\oo{j}})}{f^*\mu}\right)\\
    &=  \lambda \omega_0 + f^* \Ric \theta - f^*\omega_{WP,\lambda}.
\end{align*}
\end{proof}

\begin{proposition}\label{WPL_FS_equation}
On $B\backslash S'$ we have
\begin{equation*}
    -\Ric(f_*\Omega) + \omega_{WP,\lambda} = \frac{1}{1-e^{-T}}\eta.
\end{equation*}
\end{proposition}

\begin{proof}
We have
\begin{align*}
    i\d\dbar\log f_* \Omega = i\d\dbar\log G' - \Ric \eta.
\end{align*}
But note that on $X\backslash S$ the two previous results, Lemma \ref{pullback_G'_is_G} and Proposition \ref{ric_of_fibration_volume_form}, as well as the choice of $\Omega$, (\ref{choice_of_Omega}), give
\begin{align*}
    f^*i\d_B\dbar_B \log G' &= i\d_X\dbar_X \log f^*G'\\
    &= i\d_X\dbar_X\log G\\
    &= \frac{1}{1-e^{-T}}f^*\eta - \lambda \omega_0 - i\d_X \dbar_X \log(\omega_{SPR}^{n-m} \wedge f^*\eta^m)\\
    &= \frac{1}{1-e^{-T}} f^*\eta + f^*\Ric\eta - f^*\omega_{WP,\lambda}.
\end{align*}
Then as $f$ is a submersion on $X \backslash S$ we deduce
\begin{equation*}
    i\d\dbar\log G' = \frac{1}{1-e^{-T}} \eta + \Ric \eta - \omega_{WP,\lambda}
\end{equation*}
and the result follows.
\end{proof}

\noindent Now we arrive at the main theorem.

\begin{theorem}
The base space metric $\omega_B$ satisfies the following twisted K\"ahler-Einstein equation on $B\backslash S'$
\begin{equation}\label{twisted_KE}
    \Ric \omega_B = -\omega_B -\lambda \eta + \omega_{WP,\lambda}.
\end{equation}
\end{theorem}

\begin{proof}
From the complex Monge-Ampere equation for $\omega_B$, (\ref{base_space_CMA}), we see that
\begin{align*}
    \Ric \omega_B &= -i\d\dbar  \log G' -i\d\dbar \rho_B +\Ric \eta.
\end{align*}
Combining this with the equation
\begin{align*}
    -i\d\dbar \log G' = -\frac{1}{1-e^{-T}} \eta - \Ric\eta + \omega_{WP,\lambda}
\end{align*}
from the proof of the previous result, Prop. \ref{WPL_FS_equation}, we have the desired result.
\end{proof}

\noindent Consequently, we can find an expression for the Ricci curvature of the volume form $\Omega_{\omega_{SPR},\omega_B}$ defined below. This can be useful in reducing the K\"ahler-Ricci flow to a complex Monge-Ampere flow when the map $f:X \to B$ is a submersion, since one needs to choose a suitable volume form using the cohomological equation $[f^*\omega_B] = [f^*\eta] = [\omega(T)]$, dependent on the reference metric.

\begin{corollary}
The volume form
\begin{equation*}
    \Omega_{\omega_{SPR},\omega_B} = {n \choose m} \omega^{n-m}_{SPR} \wedge f^*\omega_B^m,
\end{equation*}
satisfies the following equation on $X \backslash S$,
\begin{align}\label{base_metric_1_volume_form_ric_curvature}
    f^*\omega_B = e^{-T} \omega_{SPR}  - (1-e^{-T})\Ric(e^{\lambda(f^*\rho_B-\rho_{SPR})} \Omega_{\omega_{SPR},\omega_B}).
\end{align}
\end{corollary}

\begin{proof}
By the application of Prop. \ref{ric_of_fibration_volume_form} and the twisted K\"ahler-Einstein equation, (\ref{twisted_KE}), we have
\begin{align*}
    \Ric \Omega_{\omega_{SPR},\omega_B} &= \lambda \omega_0 + f^* \Ric \omega_B - f^* \omega_{WP,\lambda}\\
    &= \lambda \omega_0 - f^*\omega_B - \lambda f^*\eta.
\end{align*}
Therefore,
\begin{align*}
    f^*\omega_B &= e^{-T} \omega_{SPR} + e^{-T}(\omega_0 - \omega_{SPR} + f^*\omega_B - f^*\eta) - (1-e^{-T})\Ric \Omega_{\omega_{SPR},\omega_B}\\
    &= e^{-T} \omega_{SPR} + e^{-T}i\d\dbar(f^*\rho_B - \rho_{SPR}) - (1-e^{-T}) \Ric \Omega_{\omega_{SPR},\omega_B}\\
    &=e^{-T} \omega_{SPR}  - (1-e^{-T})\Ric(e^{\lambda(f^*\rho_B-\rho_{SPR})} \Omega_{\omega_{SPR},\omega_B}).
\end{align*}
\end{proof}

\noindent We make the observation that
\begin{equation*}
    f^*\omega_B = e^{-T} \omega_{SPR}  - (1-e^{-T}) \Ric \Omega_{\omega_{SPR},\omega_B}
\end{equation*}
holds on $X \backslash S$ if and only if
\begin{equation*}
    \omega_{SPR} - \omega_0 = f^*\omega_B - f^*\eta
\end{equation*}
holds on $X \backslash S$ or if 
\begin{equation*}
    \rho_{SPR} - f^*\rho_B = C.
\end{equation*}
This is a seemingly strong condition but is of interest in \cite{B25}. At the minimum, it implies that $(X_b,\omega_{0,b})$ is K\"ahler-Einstein for all $b \in B \backslash S'$, i.e. $\omega_{SPR,b} = \omega_{0,b}$ and so $\Ric \omega_{0,b} = \lambda \omega_{0,b}$.\\

\noindent Actually, the appearance of $-\lambda \eta$ in the twisted K\"ahler-Einstein equation is superficial. If we instead consider the following complex Monge-Ampere equation on $B$,
\begin{equation*}
    \left(\frac{1}{1-e^{-T}}\eta + i\d\dbar \rho_B'\right)^m =  G' e^{\rho_B'} \left(\frac{1}{1-e^{-T}}\eta\right)^m
\end{equation*}
then, again, for some $\delta, \varepsilon > 0 $ such that $0 < \delta \leq G' \in L^{1+\e}(B,\frac{1}{1-e^{-T}}\eta)$ and, by Remark \ref{base_space_CMA_remark}, the complex Monge-Ampere equation has a unique solution $\rho_B' \in \PSH(B,\frac{1}{1-e^{-T}}\eta) \cap C^0(B) \cap C^{\infty}(B\backslash S')$ where
\begin{equation*}
    \omega'_B := \frac{1}{1-e^{-T}}\eta + i\d\dbar \rho_B'
\end{equation*}
is also a K\"ahler metric. The same calculations shows that, in fact,
\begin{equation*}
    \Ric \omega'_B = -\omega'_B + \omega_{WP,\lambda}.
\end{equation*}

\begin{remark}
One may be surprised that despite working with a Fano fibration we can construct a singular K\"ahler metric on the base space that satisfies the negative twisted K\"ahler-Einstein equation. However, only the fibres are Fano and it seems to be a consequence of the geometric setting that the base is not necessarily so, additionally, the positivity of the $(1,1)$-form $\omega_{WP,\lambda}$ is unknown. 
\end{remark}

\noindent Furthermore, we have
\begin{equation}\label{base_metric_2_volume_form_ric_curvature}
    (1-e^{-T})f^*\omega_B' = e^{-T} \omega_{SPR} - (1-e^{-T}) \Ric\left(e^{-\lambda \rho_{SPR}} \Omega_{\omega_{SPR},\omega'_B}\right)
\end{equation}
on $X \backslash S$, where the $e^{-\lambda \rho_{SPR}}$ term vanishes if and only if $\omega_{SPR} = \omega_0$ holds on $X \backslash S$ or if $\rho_{SPR} = C$. Again, this implies that $(X_b,\omega_{0,b})$ is K\"ahler-Einstein for all $b \in B \backslash S'$, since it follows that $\omega_{SPR,b} = \omega_{0,b}$ and $\Ric \omega_{0,b} = \lambda \omega_{0,b}$. The condition $\omega_{SPR} = \omega_0$ is of interest in the second paper \cite{B25}.\\
Finally, it is natural to ask what the relation is, if any, between $\omega_B$ and $\omega_B'$.

\section{Cohomological Consequences for a Submersion}

\noindent Let us assume in this section that the map $f: X \to B$ is now a submersion of compact K\"ahler manifolds, i.e. $S' = \varnothing$. In this case, the semi-prescribed Ricci curvature form $\omega_{SPR}$, $(1,1)$-form $\omega_{WP,\lambda}$ and base space metrics $\omega_B$, are globally defined smooth forms on the base $B$ and we are able to consider the cohomology of the twisted K\"ahler-Einstein equation (\ref{twisted_KE}), noting that $d \omega_{WP,\lambda} = 0$ by construction. One finds
\begin{equation*}
    -2\pi c_1(B) + [\omega_{WP,\lambda}] = \frac{1}{1-e^{-T}}[\eta].
\end{equation*}
Alternatively, one can take the cohomology class of equation (\ref{WPL_FS_equation}), as $f_*\Omega$ is now becomes a global volume form on $B$, or the twisted K\"ahler-Einstein equation for $\omega_B'$.\\ 

\noindent We also present an algebro-geometric proof.

\begin{theorem}\label{base_space_cohomology}
Let $f: X \to B$ be a submersion of compact K\"ahler manifolds, then
\begin{equation*}
    -2\pi c_1(B) + [\omega_{WP,\lambda}] = \frac{1}{1-e^{-T}}[\eta].
\end{equation*}
\end{theorem}

\begin{proof}
Recall the associated relative bundle, $D_{X/B}$, is defined to be the $\Q$-line bundle
\begin{equation*}
    D_{X/B} = \frac{1}{1-e^{-T}} D - f^* K_B.
\end{equation*}
It follows that
\begin{equation*}
    D^k = D_{X/B}^{k'} \otimes f^*( K_B^{k'})
\end{equation*}
since $k' = k(1-e^{-T})$. Then, as the map $f$ has connected fibres, i.e. $f_* \OO_X = \OO_B$, we know by the projection formula for direct images, \cite[Ex. 5.1]{H77}, that $f_* f^* \OO_B(1) = \OO_B(1)$ and
\begin{equation*}
    \OO_B(1) = f_* D^k = f_* D_{X/B}^{k'} \otimes K_B^{k'}.
\end{equation*}
Therefore
\begin{equation*}
    2\pi c_1(\OO_B(1)) = 2\pi c_1(f_* D_{X/B}^{k'}) - 2\pi k' c_1(B)
\end{equation*}
and the result follows by using the fact that
\begin{equation*}
    \frac{2\pi}{k'} c_1(f_* D_{X/B}^{k'}) = [\Ric h_{WP,\lambda}] = [\omega_{WP,\lambda}]
\end{equation*}
and $[\eta] = \frac{2\pi}{k} c_1(\OO_B(1))$ and $k' = k(1-e^{-T})$.
\end{proof}

\noindent We immediately get the following.

\begin{corollary}
If $f:X \to B$ is a submersion of compact K\"ahler manifolds then
\begin{equation*}
    2\pi c_1(X) = \lambda [\omega_0] + 2\pi f^* c_1(B) - [f^*\omega_{WP,\lambda}].
\end{equation*}
\end{corollary}

\begin{proof}
Apply the previous result to
\begin{equation*}
    [f^*\eta] = [\omega(T)] = e^{-T} [\omega_0] - (1-e^{-T}) 2\pi c_1(X).
\end{equation*}
\end{proof}

\section{Constructions for a K\"ahler-Einstein Fibration}

\noindent Let us assume again that $f: X \to B$ may have singular fibres. Recall that the conditions $\omega_{SPR} - \omega_0 = f^*\omega_B - f^*\eta$ and $\omega_{SPR} = \omega_0$ appeared to be of interest earlier in relation to equations (\ref{base_metric_1_volume_form_ric_curvature}) and (\ref{base_metric_2_volume_form_ric_curvature}) for the purpose of extracting some nice behaviour for the K\"ahler-Ricci flow via a nice choice of volume form for the complex Monge-Ampere equation. Unfortunately, these conditions are particularly strong, they imply that the initial metric needs to restrict to a K\"ahler-Einstein metric on each fibre. If we instead assume the fibres admit a smoothly varying family of K\"ahler-Einstein metrics, akin to $\omega_{SPR}$ or the semi-Ricci flat form $\omega_{SRF}$, which do not necessarily need to be the family $(\omega_{0,b})_{b \in B \backslash S'}$, then we can still extract the same behaviour for the K\"ahler-Ricci flow, \cite[Thm. 1.2]{B25}. However, one must alter the constructions of the $(1,1)$-form $\omega_{WP,\lambda}$ and twisted K\"ahler-Einstein metrics $\omega_{B}$, $\omega_{B'}$. We shall summarize the differences in this section.\\

\noindent To be precise, we assume there exists a $p > 1$ and a $\rho_{SKE} \in C^{\infty}(X\backslash S)$ such that $e^{-\lambda \rho_{SKE}} \in L^p(X,\Omega)$ for some fixed volume form $\Omega$ and $\rho_{SKE}$ can have a logarithmic pole of a small order from above along $S$, (cf. \cite[\S 3.1]{BBEGZ19}, \cite{CSYZ22}). We denote $\omega_{SKE}$ to be the $d$-closed, real $(1,1)$-form $\omega_{SKE} = \omega_0 + i\d\dbar \rho_{SKE}$ on $X \backslash S$, and require that its restriction to each fibre, $X_b$, is a K\"ahler metric denoted $\omega_{SKE,b} := \omega_{SKE}|_{X_b}$ which satisfies
\begin{equation*}
    \Ric \omega_{SKE,b} = \lambda \omega_{SKE,b}.
\end{equation*}
That is, $\omega_{SKE,b}$ is K\"ahler-Einstein and we shall refer to $\omega_{SKE}$ as the semi K\"ahler-Einstein form. We note that the extra regularity assumptions on $\rho_{SKE}$ are required in order to construct suitable base space metrics $\omega_B$, $\omega_B'$.\\

\noindent First, we construct a smooth Hermitian metric $h_{SKE}$ for $L$ over $X\backslash S$ with fibre-wise curvature $\omega_{SKE,b}$. Let us call it the semi K\"ahler-Einstein Hermitian metric for $L$. Previously, as the initial line bundle $L$ satisfies $2\pi c_1(L) = [\omega_0]$, we chose $h_L$ to be a Hermitian metric such that $\Ric h_L = \omega_0$ and it was unique up to a multiplicative constant, i.e. for any other Hermitian metric $h'_L$ such that $\Ric h'_L = \omega_0$, $h'_L = C h_L$ for some $C > 0$. We define $h_{SKE}$ using $h_L$ by
\begin{equation*}
    h_{SKE} = e^{-\rho_{SKE}} h_L
\end{equation*}
so that for all $b \in B \backslash S'$
\begin{equation*}
    \Ric h_{SKE,b} = \Ric h_L|_{X_b}  + i\d\dbar \rho_{SKE,b} = \omega_{SKE,b},
\end{equation*}
where $h_{SKE,b}$ denotes the restriction of the $h_{SKE}$ to the line bundle $L|_{X_b}$.\\
We note that the Hermitian metric $h_{SKE}$ is unique up to a multiplicative constant in the same manner as $h_L$ but, as before, this will not affect the definition of $\omega_{WP,\lambda}$ and simply scales the family of volume forms $\Omega_{\Psi_b}$ by the same fixed constant.\\

\noindent Now the semi K\"ahler-Einstein Hermitian metric, $h_{SKE}$, essentially takes the role of $h_{L}$, and $\omega_{SKE}$ replaces $\omega_{SPR}$. For the $(1,1)$-form $\omega_{WP,\lambda}$, take a holomorphically varying family $\Psi_b \in H^0(X_b,D^k_b)$ of global, non-vanishing, holomorphic sections for $b \in U'$ for some open set $U' \subseteq B$. Let us write $\Psi_b$ in local product coordinates in some open set $U \subseteq X \backslash S$ and using some holomorphic frame field $s$ for $L$ over $U$ as
\begin{equation*}
    \Psi_b = F(b,z) s^\alpha \otimes (dz^1 \wedge \dots \wedge dz^{n-m})^\beta.
\end{equation*}
Define the family of volume forms $\Omega_{\Psi_b}$ by
\begin{equation*}
    \Omega_{\Psi_b} = i^{n-m} |F(b,z)|^{\frac{2}{\beta}}|s|^{\frac{2\alpha}{\beta}}_{h_{SKE}} dz^1 \wedge \dots \wedge d\oo{z}^{n-m},
\end{equation*}
so that, as a result,
\begin{equation*}
    \Ric \Omega_{\Psi_b} = \lambda \omega_{SKE,b},
\end{equation*}
and we can define the $(1,1)$-form $\omega_{WP,\lambda}$ locally in $U'$ by 
\begin{equation*}
    \omega_{WP,\lambda}|_{U'} = - i\d\dbar\log \int_{X_b} \Omega_{\Psi_b}.\\
\end{equation*}

\noindent Furthermore, to construct the twisted K\"ahler-Einstein metrics, $\omega_{B}$, $\omega'_B$, one instead considers the real, positive function $G'$ which is smooth on $B \backslash S'$ defined by
\begin{equation*}
    G' = \frac{f_*\Omega'}{V\eta^m}
\end{equation*}
where $V = {n\choose m} \int_{X_b} \omega_{0,b}^{n-m}$ and the volume form $\Omega'$ can be chosen to satisfy
\begin{equation*}
    f^*\eta = e^{-T} \omega_{SKE} - (1-e^{-T})\Ric \Omega',
\end{equation*} 
on $X \backslash S$ by taking $\Omega' = e^{-\lambda \rho_{SKE}}\Omega$ where $\Omega$ is as before, (\ref{choice_of_Omega}).\\
In order to apply Remark \ref{base_space_CMA_remark} to construct the singular K\"ahler metrics on the base, we need to prove that $G'$ satisfies $0 < \delta \leq G' \in L^{1+\e}(B,\eta)$ for some $\delta, \e > 0$. This can be done by adapting the arguments of \cite{ST12} and Lemma \ref{pullback_G'_is_G} in the following way. First we define the real, smooth, positive function on $X\backslash S$
\begin{equation*}
    G = \frac{\Omega'}{{n \choose m} \omega_{SKE}^{n-m} \wedge f^*\eta^m}
\end{equation*}
which is constant along fibres and satisfies $G = f^* G'$ on $X \backslash S$. Next, we observe
\begin{align*}
    \int_B G'^{1+\e} \eta^m &= \frac{1}{V}\int_B G'^{\e} f_*\Omega' = \frac{1}{V}\int_X (f^* G')^{\e} \Omega'\\
    &= \frac{1}{V} \int_X e^{-(1+\e)\lambda\rho_{SKE}}\left(\frac{\Omega}{{n\choose m} \omega_{SKE}^{n-m} \wedge f^*\eta^m}\right)^\e \Omega \\
    &\leq \frac{1}{V}\left(\int_X e^{-q(1+\e)\lambda \rho_{SKE}} \Omega\right)^{1/q} \left(\int_X \left(\frac{\Omega}{{n\choose m} \omega_{SKE}^{n-m} \wedge f^*\eta^m}\right)^{r\e}\Omega\right)^{1/r}\\
    &= \frac{1}{V}\left(\int_X e^{-(q+(q-1)\e')\lambda \rho_{SKE}} \Omega\right)^{1/q} \left(\int_X \left(\frac{\Omega}{{n\choose m} \omega_{SKE}^{n-m} \wedge f^*\eta^m}\right)^{\e'}\Omega\right)^{1/r}
\end{align*}
where $\frac{1}{q}+\frac{1}{r} = 1$ and we set $\e' = r\e$. We can choose $\e' > 0$ sufficiently small such that the second integral in the last line is bounded by the same argument as \cite[Prop. 3.2]{ST12}. Then we set $q = \frac{p + \e'}{1+\e'} > 1$ so that $p = q + (q-1)\e'$ and we can apply the fact $e^{-\lambda \rho_{SKE}} \in L^p(X,\Omega)$. That is, there exists some $\e > 0$ such that $G' \in L^{1+\e}(B,\eta)$.\\
Next, note that since $\int_{X_b} \omega_{SKE,b}^{n-m} = \int_{X_b} \omega_{0,b}^{n-m}$ there exists a point $x \in X_b$ such that $\omega_{SKE,b}^{n-m}(x) = \omega_{0,b}^{n-m}(x)$ and so there exists a $\delta > 0$ such that
\begin{align*}
    G'(b) = G(x) &= \left(\frac{e^{-\lambda \rho_{SKE}}\Omega}{{n \choose m} \omega_{0}^{n-m} \wedge f^* \eta^m}\right)(x) \frac{\omega_{SKE,b}^{n-m}(x)}{\omega_{0,b}^{n-m}(x)} \geq \delta
\end{align*}
for all $b \in B \backslash S'$ if $\inf_{X\backslash S} \frac{e^{-\lambda \rho_{SKE}}\Omega}{{n \choose m} \omega_0^{n-m} \wedge f^* \eta^m} \geq \delta$. But, this holds since the denominator vanishes along $S$ and $\rho_{SKE}$ has a logarithmic pole from above of a sufficiently small order along $S$.\\ 
Therefore, the complex Monge-Ampere equation
\begin{equation*}
    (\eta + i\d\dbar \rho_B)^m = G' e^{\rho_B} \eta^m
\end{equation*}
has a unique solution $\rho_B \in \PSH(B,\eta) \cap C^0(B) \cap C^{\infty}(B\backslash S')$ such that
\begin{equation*}
    \omega_B = \eta + i\d\dbar \rho_B
\end{equation*}
is a smooth K\"ahler-metric on $B\backslash S'$.\\
We note that in order to derive the twisted K\"ahler-Einstein equation for $\omega_B$ one now needs the following version of Prop. \ref{ric_of_fibration_volume_form} where $\omega_{WP,\lambda}$ here below is, of course, the $\omega_{WP,\lambda}$ for the $\omega_{SKE}$ case.

\begin{proposition}
For any K\"ahler metric $\theta $ on $B\backslash S'$, the following equation holds on $X \backslash S$
\begin{equation*}
    \Ric(\omega_{SKE}^{n-m}\wedge f^*\theta^m) = \lambda \omega_{SKE} + f^*\Ric \theta - f^*\omega_{WP,\lambda}.
\end{equation*}
\end{proposition}

\noindent Hence,
\begin{theorem}
The singular K\"ahler metric $\omega_B$ satisfies the following twisted K\"ahler-Einstein equation
\begin{equation*}
    \Ric \omega_B = -\omega_B - \lambda \eta + \omega_{WP,\lambda}
\end{equation*}
on $B \backslash S'$.
\end{theorem}

\noindent Additionally, the volume form $\Omega_{\omega_{SKE}, \omega_B}$ satisfies
\begin{equation}\label{base_metric_3_volume_form_ric_curvature}
    f^*\omega_B = e^{-T} \omega_{SKE} - (1-e^{-T})\Ric(e^{\lambda f^*\rho_B} \Omega_{\omega_{SKE}, \omega_B})
\end{equation}
on $X \backslash S$ where the $e^{\lambda f^*\rho_B}$ term vanishes if and only if $f^*\omega_B = f^*\eta$, or if $f^*\rho_B = C$.\\ 

\noindent Finally, we can consider the complex Monge-Ampere equation
\begin{equation*}
    \left(\frac{1}{1-e^{-T}} \eta + i\d\dbar \rho_B'\right)^m =  G'e^{\rho_B'} \left(\frac{1}{1-e^{-T}}\eta\right)^m,
\end{equation*}
which has a unique solution $\rho_B' \in \PSH\left(B,\frac{1}{1-e^{-T}}\eta\right) \cap C^0(B) \cap C^\infty(B\backslash S')$ such that
\begin{equation*}
    \omega'_B = \frac{1}{1-e^{-T}}\eta + i\d\dbar \rho_B'
\end{equation*}
is a smooth K\"ahler metric on $B\backslash S'$ since $0 < \delta \leq G' \in L^{1+ \e}\left(B, \frac{1}{1-e^{-T}}\eta\right)$ for some $\delta, \e > 0$. This metric satsifies the twisted K\"ahler-Einstein equation,
\begin{theorem}
We have the following on $B \backslash S'$,
\begin{equation*}
    \Ric \omega'_B = -\omega'_B + \omega_{WP,\lambda}.
\end{equation*}
\end{theorem}
\noindent And the volume form $\Omega_{\omega_{SKE},\omega_B'}$ always satisfies
\begin{equation}\label{base_metric_4_volume_form_ric_curvature} 
    (1-e^{-T})f^*\omega'_B = e^{-T}\omega_{SKE} - (1-e^{-T})\Ric \Omega_{\omega_{SKE},\omega'_B}
\end{equation}
on $X \backslash S$.

\end{document}